\documentclass[11pt]{amsart}

\usepackage{hyperref}
\usepackage{amsfonts,mathrsfs,bbm,latexsym,rawfonts,amsmath,amssymb,amsthm,epsfig,xypic}
\usepackage{fullpage, setspace, color}

\newtheorem{thm}{Theorem}[section]
\newtheorem{cor}[thm]{Corollary}
\newtheorem{rem}[thm]{Remark}
\newtheorem{lem}[thm]{Lemma}

\numberwithin{equation}{section}

\newcommand{\ba}{\begin{aligned}}
\newcommand{\ea}{\end{aligned}}

 \newcommand{\al}{\alpha}
 \newcommand{\be}{\beta}
 
 \newcommand{\Ld}{\Lambda}
 \newcommand{\de}{\delta}
 
 \newcommand{\vep}{\varepsilon}
 \newcommand{\Si}{\Sigma}

 \newcommand{\Om}{\Omega}
 \newcommand{\ga}{\gamma}
 \newcommand{\Ga}{\Gamma}

 \newcommand{\ep}{\epsilon}

 \newcommand{\sch}{Schr\"odinger }

 \newcommand{\E}{\mathcal{E}}
 
 \newcommand{\G}{\mathcal{G}}
 \renewcommand{\P}{\mathcal{P}}
 \newcommand{\Q}{\mathcal{Q}}
 \newcommand{\R}{\mathcal{R}}
 
 \renewcommand{\L}{\mathcal{L}}
 \renewcommand{\S}{\mathscr{S}}

 \newcommand{\N}{\mathcal{N}}
 
 \renewcommand{\H}{\mathcal{H}}
 \newcommand{\J}{\mathcal{J}}

 \DeclareMathOperator{\tr}{tr}
 \DeclareMathOperator{\vol}{vol}
 
 \DeclareMathOperator{\Hom}{Hom}

 \newcommand{\Real}{\mathbb{R}}

 \newcommand{\norm}[1]{\Vert#1\Vert}
 
 \def\<{\left\langle} \def\>{\right\rangle}
 \def\({\left(} \def\){\right)}
 \newcommand{\n}{\nabla}
 \newcommand{\p}{\partial}

 \renewcommand{\t}[1]{\tilde{#1}}
 \renewcommand{\b}[1]{\bar{#1}}

 \newcommand{\bn}{\bar{\n}}
 \newcommand{\tn}{\tilde{\nabla}}

\begin{document}

\title{Local Existence and Uniqueness of Skew Mean Curvature Flow}

\author{Chong Song}

\address{School of Mathematical Sciences, Xiamen University, Xiamen, 361005, P.R.China.}
\email{songchong@xmu.edu.cn}


\date{\today}

\begin{abstract}
  The Skew Mean Curvature Flow(SMCF) is a Schr\"odinger-type geometric flow canonically defined on a co-dimension two submanifold, which generalizes the famous vortex filament equation in fluid dynamics. In this paper, we prove the local existence and uniqueness of general dimensional SMCF in Euclidean spaces.
\end{abstract}

\maketitle

\tableofcontents


\section{Introduction}

\subsection{Definition and backgrounds}

The Skew Mean Curvature Flow(SMCF) is a Schr\"odinger type geometric flow defined on a co-dimension two submanifold, which evolves the submanifold along its bi-normal direction with the speed given by its mean curvature. More specifically, suppose $(\b{M},\b{g})$ is an $(n+2)$ dimensional oriented Riemannian manifold, $\Si$ is an $n$ dimensional oriented manifold, then the SMCF is a family of time-dependent immersions $F:[0,T)\times \Si\to \b{M}$ satisfying
\begin{equation}\label{e:SMCF}
\p_tF=J(F)H(F).
\end{equation}
Here $H(F)$ is the mean curvature vector field of the submanifold and $J(F)$ is the induced complex structure on its normal bundle $N\Si$. Namely, $J$ rotates a normal vector $\nu$ by $\pi/2$ positively in the normal plane, such that for any oriented basis $\{\ep_1,\cdots,\ep_n\}\subset T_x\Si$, $F_*(\ep_1)\wedge F_*(\ep_2)\wedge\cdots\wedge F_*(\ep_n)\wedge \nu\wedge (J\nu)$ coincides with the orientation of $\b{M}$.

The simplest model of SMCF is the one dimensional SMCF in three dimensional Euclidean space, which reduces to the famous Vortex Filament Equation(VFE)
\[ \p_t\ga=\ga_s\times\ga_{ss}=\kappa \textbf{b}.\]
Thus the SMCF is a geometric generalization of the VFE both in higher dimensions and in general ambient Riemannian manifolds. It naturally arises in both the context of superfluid~\cite{Jerrard-GP-2002} and classical hydrodynamics~\cite{Shashikanth-JMP-2012, Khesin-SMCF-2012}, describing the locally-induced motion of a co-dimension two vortex membrane. For a detailed introduction to the backgrounds and motivations to study the SMCF, we refer to our previous paper \cite{SS-SMCF-2019}.

The SMCF falls in the category of dispersive, or more precisely, \sch type geometric flows. Probably due to lack of analytic tools, the research of \sch type geometric flows are rather underdeveloped as compared to their parabolic (or hyperbolic) counterparts. Little is known about the SMCF except for the 1 dimensional VFE, which bears very rich structures.

It is well-known that the VFE is equivalent to a standard cubic \sch equation by the Hasimoto transformation, hence is completely integrable and behaves nicely.  For example, the initial value problem of VFE is globally well-posed for smooth initial curves, and has a global weak solution which enjoys a weak-strong uniqueness property for integral currents as initial data~\cite{JS-JEMS-2015}. Similarly, the one dimensional SMCF in a general three dimensional Riemannian manifold is also equivalent to a \sch equation and exists globally~\cite{Gomez-Thesis-2004}. However, higher dimensional SMCFs ($n\ge 2)$ are essentially different from the VFE and such a magical transformation seems impossible. Very recently, Khesin and Yang~\cite{Khesin-Yang-2019} provide an evidence from the point of view of integrable systems. They also construct an interesting example of the product of different dimensional spheres where the SMCF blows up in finite time.

From the perspective of PDE analysis, there are two main difficulties in the investigation of a high dimensional ($n\ge 2$) SMCF. First of all, for a general (non-graphical) high dimensional submanifold, there does not exist a global coordinate system or a global gauge on its normal bundle. Thus many powerful analytical tools in the study of dispersive equations, such as harmonic analysis, are invalid. Secondly, even if we write down the equation in local coordinates, the SMCF still seems very challenging. In fact, it is well-known that the mean curvature $H(F)$ is a quasi-linear and degenerate elliptic operator on $F$. In local coordinates $H(F)$ can be written as
\[ H(F)^\al=(\Delta_gF)^\al=g^{ij}(\p_i\p_jF^\al-\Ga_{ij}^k\p_kF^\al+\b{\Ga}_{\be\ga}^\al\p_iF^\be\p_jF^\ga),\]
where $g_{ij}=\b{g}(\p_iF,\p_jF)$ is the induced metric, $\Ga_{ij}^k$ and $\b{\Ga}_{\be\ga}^\al$ are the Christoffel symbols of $(\Si,g)$ and $(\b{M},\b{g})$ respectively. In particular, $\Ga_{ij}^k$ involve with the first order derivatives of $g_{ij}$ and hence second order derivatives of $F$. Moreover, the complex structure $J(F)$ also involves with the first order derivatives of $F$. Thus the high dimensional SMCF (\ref{e:SMCF}) can be regarded as a quasi-linear \sch type system that is defined on a manifold.

As far as the author knows, there are only a few results on the existence and uniqueness of general quasi-linear \sch equations on Euclidean spaces, usually under strong restrictions on the non-linear structures. See for example \cite{KPV-quai-linear-Schrodinger-2004}, the recent book~\cite{LP-book-2015} and references therein.

On the other hand, during the last twenty-years, there has been remarkable developments on the research of another typical \sch type flow, namely, the celebrated \sch map flow. Suppose $(M,g)$ is a Riemannian manifold and $(N, \omega)$ is a symplectic manifold with an almost complex structure $J_N$, then the \sch map flow $u:[0,T)\times M\to N$ is defined by the equation
\begin{equation}\label{e-schrodinger-map-flow}
  \p_tu=J_N(u)\tau(u),
\end{equation}
where $\tau(u)$ is the tension field of $u$. The local well-posedness, finite time blow-up and global behavior of \sch map flow are extensively studied, see for example \cite{Ding-ICM-2002, RRS-Global-2009, BIKT-Global-2011, MRR-blowup-2013, Li-Global-2018} and references therein. Compared to the SMCF, the \sch map flow is considerably easier since the corresponding operator $\tau(u)$ is semi-linear and elliptic. However, there is a strong connection between the SMCF and the \sch map flow.

In \cite{Song-Gauss-2017}, we found that the Gauss map of a SMCF satisfies a \sch map flow into a Grassmannian manifold with evolving metric. Indeed, the Gauss map of a SMCF in an Euclidean space is a map $\rho:[0,T)\times (\Si, g)\to G(n,2)$ mapping to the Grassmannian manifold $G(n,2)$, which is a K\"ahler manifold. If we denote the canonical complex structure on $G(n,2)$ by $\J$, then $\rho$ satisfies
\begin{equation}\label{e:Gauss}
\p_t\rho=\J(\rho)\tau_g(\rho).
\end{equation}
The key difference between (\ref{e:Gauss}) and ordinary \sch map flow (\ref{e-schrodinger-map-flow}) is that now the metric $g$ of the underlying manifold is also evolving along time, by the equation
 \[ \p_tg=-2\<JH,A\>,\]
Thus the SMCF can be regarded as a coupled \sch map flow, which is more sophisticated, and new methods must be deployed.

In \cite{SS-SMCF-2019}, the author and Sun proved the local existence of two dimensional SMCF in $\Real^4$ by applying a parabolic regularization and an energy method. The proof essentially relies on an uniform estimate of the second fundamental form of a two dimensional surface, which is obtained  by using blow-up techniques in geometric analysis. However, it seems that the key estimate only holds for two dimensional surfaces. In this paper, we will prove both local existence and uniqueness of SMCF in general dimensions.

\subsection{Main results}


Suppose $\Si_0$ is an $n$-dimensional compact oriented manifold where $n\ge 2$. Given a co-dimensional two immersion $F_0:\Si_0\to \Real^{n+2}$, we consider the initial value problem
\begin{equation}\label{e:initial-value-problem}
\left\{
\ba
&\p_tF=JH,\\
&F(0,\cdot)=F_0(\cdot).
\ea
\right.
\end{equation}

Let $k_0=[\frac{n}{2}]+1$ where $[\frac{n}{2}]$ denotes the integer part of $n/2$. Fix a small number $\de\in(0, 1)$ and let $p=n+\de$. For $k\ge k_0$, we define the energy of an immersion $F$ by
\[ \E_k(F)=\vol+\norm{H}_p^2+\norm{A}_{k,2}^2,\]
where $\norm{A}_{k,2}$ is the $H^{k,2}$-Sobolev norm with regard to the normal connection $\n$ on the product bundle $N\Si\otimes(T^*\Si)^s$, namely,
\[ \norm{A}_{k,2}:=\norm{A}_{H^{k,2}}=\(\sum_{l=0}^k\int |\n^l A|^2\)^{1/2}.\]

Our main result is the following theorem.

\begin{thm}\label{t:main-existence}
Suppose $\Si_0$ is an $n$-dimensional compact oriented manifold with $n\ge 2$ and $F_0:\Si_0\to \Real^{n+2}$ is a smooth immersion, then the initial value problem (\ref{e:initial-value-problem}) has a unique smooth solution $F:[0,T)\times \Si_0\to \Real^{n+2}$, where $T$ only depends on the energy $\E_{k_0}(F_0)$ of the initial immersion $F_0$.

Moreover, there exists a constant $C_k$ only depending on $k$ and $\E_{k_0}(F_0)$ such that for all $t\in [0,T)$,
\begin{equation}\label{e:bound-SSF}
  \E_k(F(t))\le C_k \E_k(F_0), \quad \forall k\ge k_0.
\end{equation}
\end{thm}

\begin{rem}
\begin{enumerate}
  \item With some efforts, one can generalize the above results to SMCFs into general ambient Riemannian manifolds with bounded geometry. But here we assume the ambient space is Euclidean space for simplicity.
  \item In the appendix, we show that the energy $\E_k$ is equivalent to the $W^{k+1,2}$-Sobolev norm of the Gauss map. See Theorem~\ref{c:equivalence} for more details.
  \item
      In \cite{DW-Schrodinger-2001}, Ding and Wang obtained the existence of a local solution to the \sch map flow with $W^{k_0+1,2}$-initial value, which seems to be optimal. In view of the fact that the Gauss map $\rho$ of the SMCF satisfies a \sch map flow (\ref{e:Gauss}),  our estimate (\ref{e:bound-SSF}) which only depends on $\E_{k_0}$, or equivalently on the $W^{k_0+1,2}$-norm of $\rho$, is consistent with Ding-Wang's result.
  \item
      It is very tempting to prove the existence of a strong solution $F\in L^\infty([0,T),W^{k+2,2})$ to the SMCF with an initial immersion $F_0\in W^{k+2,2}$. However, this can not be achieved by our estimate (\ref{e:bound-SSF}). Because the $W^{k,2}$- bound of the second fundamental form $A$ dose not directly yield a $W^{k+2,2}$-bound on the immersion $F$. There is a loss of derivative due to the diffeomorphism group of the domain manifold. Actually, we can only get a solution $F\in L^\infty([0,T),W^{k+1,2})$ in this case.
\end{enumerate}
\end{rem}

The existence part of Theorem~\ref{t:main-existence}(see Theorem~\ref{t:existence} below) is proved in Section~\ref{s:existence} by a parabolic regularization method. Namely, we first solve a perturbed SMCF which is parabolic, and then try to find a limit solution to the SMCF  by letting the perturbation vanish. The choice of the energy $\E_k$ is crucial in the proof which yields a uniform estimate for the perturbed SMCF. It is natural to put the norms of the second fundamental form $\norm{A}_{k,2}$ in the energy $\E_k$, while the term $\vol +\norm{H}_{p}$ is enclosed to ensure that the Sobolev constants in the interpolation inequalities are uniform. For more explanation of the strategy of proof, see Section~\ref{s:idea-of-proof-existence} below.

The uniqueness of the SMCF in Theorem~\ref{t:main-existence} follows from the following more general theorem. To state the result, we define the $k$-th space of immersions for $k\in \mathbb{N}$ by
\[ \S_k:=\{F:\Si_0\to \Real^{n+2}| F \text{~is an immersion with~} \norm{A}_{k,\infty}<\infty\}.\]

\begin{thm}\label{t:main-uniqueness}
Suppose $\Si_0$ is an $n$-dimensional compact oriented manifold with $n\ge 2$ and $F_0:\Si_0\to \Real^{n+2}$ is an immersion in $\S_3$ . If $F\in L^\infty([0,T],\S_2)$ and $\t{F}\in L^\infty([0,T], \S_3)$ are two solutions of the SMCF~(\ref{e:initial-value-problem}) with same initial value $F_0$, then $F=\t{F}$ a.e. on $[0,T]\times \Si_0$.
\end{thm}

\begin{rem}
Note that here we only requires one of the solutions lie in the space $L^\infty([0,T], \S_3)$ while the other solution in a weaker space $L^\infty([0,T], \S_2)$. This can be compared with the weak-strong uniqueness of the VFE, i.e. one dimensional SMCF, proved by Jerrard and Smets~\cite{JS-JEMS-2015}.
\end{rem}

The proof of Theorem~\ref{t:main-uniqueness}, which is presented in Section~\ref{s:uniqueness}, turns out to more involved than that of the existence. Our proof of uniqueness of the SMCF is based on an energy method which originates from our previous work on the uniqueness of the \sch map flow~\cite{SW-uniqueness-2018}. The main idea is simply to define an energy functional $\L=\L(F,\t{F})$ which describes the distance of two solutions, and try to show that $\L$ vanishes identically in the time span $[0,T]$. However, finding such a suitable functional $\L$ takes a lot of efforts and requires a full exploration of the geometric structure of the SMCF. In our final solution, we define $\L$ by using parallel transport on the Grassmannian manifold. It is worth mentioning that during the proof, we propose a notion of intrinsic distance between the second fundamental forms of two submanifolds, which might be of independent interest. A detailed exposition of the idea behind the construction of $\L$ is given in Section~\ref{s:idea-of-proof-uniqueness}.

\subsection{Further problems}

Finally, we propose several open problems on the well-posedness of the SMCF:
\begin{enumerate}
  \item Prove the local existence and uniqueness of a weak(or strong) solution of the SMCF for more general initial submanifolds that are less regular. This might involve a well-defined notion of weak solutions of the SMCF, cf. \cite{Jerrard-GP-2002}.
  \item Prove the existence of SMCF for complete non-compact submanifolds under suitable assumptions. In particular, we are interested in the existence of a local solution to the SMCF defined on $\Real^n$.
  \item Is the SMCF globally well-posed for initial submanifolds with certain small energy?
\end{enumerate}

\section*{Acknowledgement}

Part of this work was carried out during a visit at Tsinghua University from September 2017 to February 2018. The author would like to thank Professor Yuxiang Li, Huaiyu Jian and Hongyan Tang for their support and hospitality. He is also grateful to Professor Yu Yuan, Jingyi Chen and Youde Wang for stimulating conversations and their generous help over the past several years.


\section{Local existence of SMCF}\label{s:existence}

\subsection{Idea of proof}\label{s:idea-of-proof-existence}

To prove the local existence of the SMCF, we apply an approximating scheme by considering a perturbed SMCF
\begin{equation}\label{e:pSMCF}
\p_tF=J_\vep H:=JH+\vep H, \vep>0.
\end{equation}
For $\vep>0$, the perturbed flow (\ref{e:pSMCF}) is a linear combination of the SMCF and the Mean Curvature Flow(MCF). It turns out to be a (degenerate) parabolic system which behaves similarly as the MCF. A standard argument by using the De Turck trick guarantees the existence of a local solution $F_\vep:[0,T_\vep)\to \Real^{n+2}$ to (\ref{e:pSMCF}) for any $\vep>0$. Therefore, if we can derive a uniform estimate of $F_\vep$ and a positive lower bound of $T_\vep$, then we obtain a local solution to the SMCF (\ref{e:SMCF}) by letting $\vep\to 0$.

This approximating scheme is often referred as the parabolic regularization method, and is quite standard in dealing with the existence problem of non-linear \sch equations. But it is highly non-trivial to obtain a uniform estimate that does not depend on $\vep$, especially for \sch type geometric flows, which are quasi-linear. Note that since the parabolic term vanishes as $\vep\to0$, any parabolic type estimates, including those obtained by maximum principal, would blow-up and do not yield the desired uniform bounds.

Nevertheless, we are still able to derive a uniform bound of the second fundamental form $A_\vep$ corresponding to $F_\vep$ by using an energy method. Namely, we consider the evolution equation of the $H^{k,2}$-Sobolev norms of $A_\vep$ on $[0,T_\vep)$, and try to control the nonlinear terms by Sobolev interpolation inequalities. The main obstruction for implementing such an idea is that the Sobolev constants in the standard interpolation inequalities for tensors are not uniform, as they depend on the underlying metric, which is also evolving along the flow. To overcome this problem, we will apply a uniform Sobolev inequality (see Theorem~\ref{t-lp} below) proved by Mantagazza~\cite{Mantegazza-GAFA-2002}, which says that the Sobolev constants are uniform if $\vol+\norm{H}_p$ is uniformly bounded for some $p>n$. This motivates us to enclose the term into the total energy and define
\[ \E_k=\vol+\norm{H}_p^2+\norm{A}_{k,2}^2.\]
The point is that the Sobolev constants are uniform as long as the energy $\E_k$ is uniformly bounded.

By considering the evolution equation of $\E_k$ of $F_\vep$, we manage to derive a uniform bound on $\E_k$, which in turn gives a uniform bound of $A_\vep$. Then the convergence of $F_\vep$ to a solution of the SMCF follows by a standard argument.

\subsection{Uniform Sobolev interpolation inequalities}

In this subsection, we recall several uniform Sobolev interpolation inequalities which will be used later.

First we recall the following standard universal interpolation inequality.

\begin{thm}[Aubin~\cite{Aubin-book-1998}, Chapter 3, Section 7.6]\label{t-universal}
Suppose $M$ is a compact $m$-dimensional Riemannian manifold. Let $E$ be a vector bundle on $M$, which is endowed with a metric and a compatible connection $D$. Then for any section $s\in \Ga(E\otimes (T^*M)^p)$ and exponents $q\in [1,\infty)$ and $r\in [1,\infty]$, there is a constant $C$ only depending on the dimension $m$ and the exponents such that for all $0<j<k$
\begin{equation*}
  \norm{D^j s}_p\le C\norm{D^k s}_q^{\frac{j}{k}}\norm{s}_r^{1-\frac{j}{k}},
\end{equation*}
where
\begin{equation*}
  \frac{k}{p}=\frac{j}{q}+\frac{k-j}{r}
\end{equation*}
\end{thm}

\begin{rem}
The above theorem also applies for complete non-compact manifolds (cf.~\cite{Cantor-Sobolev-1975}).
\end{rem}

Next we have the famous Michael-Simon inequality for submanifolds.

\begin{thm}[Michael-Simon~\cite{MS-Sobolev-1973}]\label{t-MS}
Let $M$ is an immersed compact $m$-dimensional submanifold in the Euclidean space with mean curvature $H$. Then for any smooth function $u:M\to \Real^1$ and $p\in[1,m)$, we have
\begin{equation*}
  \norm{u}_{p^*}\le C(\norm{D u}_p+\norm{Hu}_p),
\end{equation*}
where $C$ is a constant only depending on the dimension $m$ and the exponent $p$.
\end{thm}

The following theorem is a quite straight forward application of Theorem~\ref{t-MS}, and provides the uniform interpolation theorem which is used in this paper. The proof involves first proving (\ref{e:interpolation}) for $a=1, j=0, k=1$ by applying Theorem~\ref{t-lp}, and then following a standard procedure as in the proof of classical Gagliardo-Nirenberg inequality.

\begin{thm}[Mantegazza~\cite{Mantegazza-GAFA-2002}]\label{t-lp}
Suppose $M$ is an $m$-dimensional compact immersed submanifold of the Euclidean space $\Real^n$. If $Vol+\norm{H}_{n+\de}\le B$ for some $\de>0$, then we have uniform Gagliardo-Nirenberg inequality for any covariant tensor $T$. In particular, there is a constant $C$ only depending on $B, m$ and the exponents such that
\begin{equation}\label{e:interpolation}
\norm{D^jT}_p\le C\norm{T}_{k,q}^a\norm{T}_r^{1-a},
\end{equation}
where $j\in[0,k], p,q, r\in[1,\infty]$ and $a\in[j/k,1]$($a\neq 1 $ if $q=m/(k-j)\neq 1$) satisfies
\[ \frac{1}{p}=\frac{j}{m}+a\(\frac{1}{q}-\frac{k}{m}\)+\frac{1-a}{r}>0.\]
In particular, when $kq>m$, we have
\[ \norm{T}_\infty\le C\norm{T}_{W^{k,q}}.\]
\end{thm}


\subsection{Evolution equations}

For a small number $\vep\ge 0$, suppose $F:[0,T]\times \Si_0\to \Real^{n+2}$ is a solution to the perturbed SMCF
\begin{equation}\label{e-pSMCF}
  \p_t F=J_\vep H=\vep H+JH,
\end{equation}
where $J_\vep = \vep I + J$ and $H$ is the mean curvature vector.

First we recall some evolution equations under the perturbed SMCF, which are obtained in Song-Sun~\cite{SS-SMCF-2019}. Note that when $\vep=0$, they reduce to evolution equations of the SMCF.

\begin{lem}\label{l:evolution-equ}
Under the perturbed SMCF~(\ref{e-pSMCF}), we have the following evolution equations
\begin{eqnarray}
                  \p_t g &=& -2\<J_\vep H, A\>, \label{e:evolution-g}\\
                  \p_t d\mu &=& -\vep|H|^2d\mu, \label{e-evo-volume}\\
                  \p_t A &=& J_\vep \Delta A + A\#A\#A,\label{e:evolution-A}\\
                  \p_t H &=& J_\vep \Delta H + A\#A\#H,\label{e:evolution-H}\\
                  \p_t \n^lA&=&J_\vep\Delta\n^l A+\sum_{i+j+k=l}\n^iA\#\n^jA\#\n^lA. \label{e:evolution-dA}
\end{eqnarray}
where $g$ is the induced metric, $A$ is the second fundamental form, $d\mu$ is the induced volume form, $\n$ and $\Delta$ is the normal connection and corresponding Laplacian operator, and $\#$ denotes linear combinations of tensors.
\end{lem}

The Gauss map of $F$ is a map $\rho:[0,T]\times(\Si_0,g)\to G$ into the Grassmannian manifold $G=G(n,2)$. Following Song~\cite{Song-Gauss-2017}, we can also deduce the evolution equation of Gauss map of the perturbed SMCF.
\begin{lem}\label{l:Gauss}
  Under the perturbed SMCF~(\ref{e-pSMCF}), the Gauss map $\rho$ satisfies
  \[ \p_t\rho = \J_\vep(\rho)\tau_g(\rho)=\vep \tau_g(\rho)+\J(\rho)\tau_g(\rho), \]
  where $\J$ is the complex structure on $G$ and $\tau_g(\rho)$ is the tension field of $\rho$.
\end{lem}

Next we derive the evolution equation of the Sobolev norms of $A$. Since by integration by parts,
\[ \int \<\n^lA,J_\vep\Delta\n^l A\>=-\vep\int|\n^{l+1}A|^2,\]
it follows from (\ref{e:evolution-dA}) that
\begin{equation*}
\p_t\norm{\n^l A}_2^2\le C\sum_{i+j+k=l}\int_M|\n^i A|\cdot|\n^j A|\cdot|\n^k A|\cdot|\n^l A|d\mu.
\end{equation*}
Applying H\"older's inequality and the universal interpolation inequality in Theorem~\ref{t-universal} with $r=\infty$, we get
\begin{lem}
Under the perturbed SMCF~(\ref{e-pSMCF}),
\begin{equation*}
  \p_t\norm{\n^l A}_2^2\le C\norm{A}_\infty^2\norm{\n^l A}_2^2.
\end{equation*}
Hence
\begin{equation}\label{e-Awk2}
  \p_t\norm{A}_{k,2}^2\le C\norm{A}_\infty^2\norm{ A}_{k,2}^2.
\end{equation}
\end{lem}

We will also need the evolution equation of the $L^p$-norm of $H$.
\begin{lem}
Under the perturbed SMCF~(\ref{e-pSMCF}), for $p\ge 2$,
\begin{equation}\label{e-Alp}
  \p_t\norm{H}_p^2 \le C (\norm{\n H}_p^2+\norm{A}_\infty^2\norm{H}_p^2).
\end{equation}
\end{lem}
\begin{proof}
By direct computation,
\begin{equation*}
  \begin{aligned}
    \p_t\norm{H}_p^p &= \int p|H|^{p-2}\<H,\n_tH\>\\
    &=p\int  |H|^{p-2}\<H,J_\ep \Delta H+A\#A\#H\>\\
    &=-p(p-2)\int |H|^{p-4}\<H,\n H\>\<H,J_\ep\n H\>-p\int |H|^{p-2}\<\n H,J_\ep \n H\>\\
    &\quad\quad +p\int  |H|^{p-2}\<H, A\#A\#H\>\\
    &\le p(p-2)\int |\n H|^2|H|^{p-2}+p\int|A|^2|H|^{p}\\
    &\le C (\norm{\n H}_p^2\norm{H}_p^{p-2}+\norm{A}_\infty^2\norm{H}_p^p).
  \end{aligned}
\end{equation*}
Thus the lemma follows.
\end{proof}

\begin{rem}
Up to now, we have only used the universal interpolation inequality in Theorem~\ref{t-universal}, thus all the constants are uniform.
\end{rem}

\subsection{Proof of existence}

Let
\[ k_0=[\frac{n}{2}]+1, \quad p_0=\frac{2n}{n-2k_0+2}.\]
Note that $k_0>n/2$, thus we have Sobolev embedding $W^{k_0,2}\hookrightarrow L^\infty$. Moreover $p_0>n$ and we have Sobolev embedding $W^{k_0,2}\hookrightarrow W^{1,p}$ for all $n<p<p_0$. (Actually, $p_0=2n$ when $n$ is odd, and $p_0=+\infty$ when $n$ is even.)

Now let $p\in (n, p_0)$ be a fixed real number. For the perturbed SMCF with $\vep>0$, we define the energy
\[\E_k(F_\vep):=\vol+\norm{H}_p^2+\norm{A}_{k,2}^2,\]
where $\vol$ is the volume and $k\ge k_0$ is an integer.

First we assume that $\norm{H}_p$ is uniformly bounded along the perturbed SMCF for some time span $[0,T_\vep]$.
Recall that by (\ref{e-evo-volume}), the volume is decreasing along the perturbed SMCF. Thus there exists a uniform constant $B>0$ such that $\vol+\norm{H}_p\le B$ for $t\in [0,T_\vep']$.

Then we can apply Theorem~\ref{t-lp} to get
\[ \norm{A}_\infty +\norm{\n A}_p\le C(B) \norm{A}_{k_0,2}.\]
It follows immediately from inequality (\ref{e-Awk2}) that
\begin{equation*}
  \p_t\norm{A}_{k,2}^2\le C(B)\norm{ A}_{k_0,2}^2\norm{ A}_{k,2}^2. \label{e-1}
\end{equation*}
Similarly, by (\ref{e-Alp}) we have
\begin{equation*}
\p_t\norm{H}_p^2 \le C (B)\norm{A}_{k_0,2}^2(1+\norm{H}_p^2).
\end{equation*}

Therefore, we conclude that

\begin{lem}\label{l:estimate-Ek}
Suppose there is a constant $B>0$, such that $\vol+\norm{H}_p\le B$ along the perturbed SMCF for a time span $[0,T_\vep']$, then for any $k\ge k_0$,
\begin{equation}\label{e:estimate-Ek}
  \p_t\E_k\le C(B)\norm{A}_{k_0,2}^2\cdot (1+\E_k).
\end{equation}
\end{lem}

The following lemma shows that there is a uniform lower bound for the time $T_\vep'$ for $\vep>0$.

\begin{lem}\label{l-uniform}
Given an initial immersion $F_0:M\to \Real^{m+2}$, there exists a uniform time $T_0>0$ only depending on $\E_0:=\E_{k_0}(F_0)$ such that
\[\E_{k_0}(t):=\E_{k_0}(F_\vep(t)) \le 2\E_0+1, \forall \vep>0, t\in [0, T_0].\]
\end{lem}
\begin{proof}
Set $B=2 (1+\E_0)$. For each $\vep>0$, define
\[T_\vep':=\sup\{T\in[0,T_\vep]|1+\E_{k_0}(t)\le B, \forall t\in [0,T]\}. \]
Clearly, $T_\vep'>0$ since the solution $F_\vep$ is smooth. Moreover, $T_\vep'$ is smaller than the maximal existence time $T_\vep$. Otherwise we can find a limit of $F_\vep(t)$ as $t\to T_\vep$ and the solution can be extended past the maximal time $T_\vep$.

Applying Lemma~\ref{l:estimate-Ek}, we get that for $t\in [0,T_\vep']$
\[ \p_t(1+\E_{k_0})\le C(B)\norm{A}_{k_0,2}^2(1+\E_{k_0})\le\bar{C}(B)(1+\E_{k_0}),\]
where $\bar{C}(B)=C(B)B$ is a constant only depending on $B$.

It follows from Gronwall's inequality that
\[ 1+\E_{k_0}(t)\le e^{\bar{C}(B)t}(1+\E_0).\]
Now by letting $t=T_\vep'$, we find $2 \le e^{\bar{C}(B)T_\vep'}$, yielding
\[ T_\vep' \ge T_0:=\frac{\ln 2}{\bar{C}(B)}.\]
\end{proof}

With the above uniform estimates, we are now in the position to prove the existence part of Theorem~\ref{t:main-existence}, which we restate as follows.

\begin{thm}\label{t:existence}
Suppose $\Si_0$ is a compact $n$ dimensional manifold and $F_0$ is a smooth immersion, then the initial value problem (\ref{e:initial-value-problem}) has a smooth solution $F:[0,T)\times \Si_0\to \Real^{n+2}$, where $T$ only depends on the energy $\E_0=\E_{k_0}(F_0), k_0=[n/2]+1$.

Moreover, there exists a constant $C_k$ only depending on $k$ and $\E_0$ such that for all $t\in [0,T)$ ,
\begin{equation}\label{e:bound-SSF1}
  \E_k(F(t))\le C_k \E_k(F_0), \forall k\ge k_0.
\end{equation}
\end{thm}
\begin{proof}
For any small $\vep>0$,  by applying DeTurck's trick, we can find a local solution $F_\vep:[0,T_\ep]\times M\to \Real^{m+2}$ to the perturbed SMCF (\ref{e-pSMCF}). (cf. Lemma 4.1 in \cite{SS-SMCF-2019})

By Lemma~\ref{l-uniform}, there exists a time $T_0$ only depending on $\E_0$ such that $T_\vep>T_0$ and
\begin{equation}\label{e:1}
  \E_{k_0}(F_\vep(t))\le 2\E_0+1, \forall \vep>0, t\in [0,T_0].
\end{equation}
Then by (\ref{e:estimate-Ek}) in Lemma~\ref{l:estimate-Ek} that for all $\vep>0$, we have uniform inequality
\begin{equation*}
  \p_t\E_k(F_\vep(t))\le C(\E_0)(1+\E_k(F_\vep(t))),
\end{equation*}
which implies
\begin{equation}\label{e:bound-Ek}
  \E_k(F_\vep(t))\le C(\E_0)\E_k(F_0)).
\end{equation}
In particular, we have
\begin{equation}\label{e:2}
  \norm{A(F_\vep(t))}_{k,2}\le C_k(\E_0)\norm{A(F_0)}_{k,2}, \quad \forall k\ge k_0.
\end{equation}

Since $\vol+\norm{H}_p$ of each submanifold $F_\vep(t)$ are uniformly bounded by (\ref{e:1}), we have uniform Sobolev inequalities guaranteed by Theorem~\ref{t-lp}. Thus it follows from (\ref{e:2}) that the second fundamental form $ A(F_\vep(t))$ and its derivatives are all uniform bounded.

Now by standard arguments, we can extract a subsequence $\vep_i\to 0$ such that $F_{\vep_i}$ converges smoothly to a limit $F_\infty:[0,T_0]\times M\to \Real^{m+2}$ which is the desired solution to the SMCF. Finally, the estimate (\ref{e:bound-SSF1}) is a direct consequence of (\ref{e:bound-Ek}).
\end{proof}

\begin{rem}
By Theorem~\ref{c:equivalence} in the appendix, the energy $\E_k$ is equivalent to the energy $\bar{\E}_k=\norm{\rho}_{W^{k+1,2}}$ of the Gauss map. The latter energy is used by Mantegazza~\cite{Mantegazza-GAFA-2002} to investigate a higher order analog of the mean curvature flow.
\end{rem}


\section{Uniqueness of SMCF}\label{s:uniqueness}

\subsection{Idea of proof}\label{s:idea-of-proof-uniqueness}

The proof of uniqueness of the SMCF is more challenging than that of existence. In fact, it is well-known that even for a (complex-valued) scalar \sch equation that contains 1st order derivatives (also referred as derivative \sch equations), the uniqueness can only be obtained under certain restrictions on the nonlinear structure (cf. \cite{LP-book-2015}). Compared to parabolic equations, this is because the leading 2nd order term in a \sch equation does not play a dominant role and provide a ``good" term which help to control the lower-order terms. Nevertheless, by fully exploring the underlying geometric structures of the SMCF, we show that the uniqueness still holds true, provided the solution is sufficiently smooth.

Suppose $F$ and $\t{F}$ are two solutions to the initial value problem (\ref{e:initial-value-problem}) with the same initial data $F_0$, we will show that the two solutions are identical by an energy method. Thus we need to define an energy functional $\L$ which describes the difference/distance of $F$ and $\t{F}$ such that $\L$ vanishes if and only if $F=\t{F}$. The main problem then is to find a suitable quantity $\L$.

One naive option is to take the distance of two submanifolds $\bar{d}(F,\t{F})$ and its derivatives, where $\bar{d}$ is the extrinsic distance in the ambient space $\bar{M}$. In our setting where $\bar{M}=\Real^{n+2}$, one may simply let $\L=\norm{F-\t{F}}_{k,2}$ for some $k\in \mathbb{N}$. But the evolution equations of derivatives of $F$ and $\t{F}$ are quite messy, because they are not geometric quantities (tensors) of corresponding submanifolds. Thus this method certainly won't work and we need to find more geometric quantities instead.

The most important geometric quantity of a submanifold is the second fundamental form. Thus one may consider the difference of the second fundamental forms, such as $\norm{A-\t{A}}_2$. Note that $A$ is a section of the bundle $\N\otimes (T^*\Si)^2$ while $\t{A}$ is a section of the bundle $\t\N\otimes (T^*\t\Si)^2$, where $\N$ and $\t\N$ are the normal bundle of $\Si=(\Si_0, g)$ and $\t\Si=(\Si_0, \t{g})$ respectively. Thus by defining the substraction $A-\t{A}$, we implicitly use the embedding $\N\subset F^*\Real^{n+2}, \t\N\subset \t{F}^*\Real^{n+2}$ and the distance of the ambient space $\Real^{n+2}$. Moreover, we directly identity $T^*\Si$ with $T^*\t\Si$ although their metrics are different.

However, by the evolution equation (\ref{e:evolution-A}), we have
\[ \p_t(A-\t{A})=J\Delta A-\t{J}\t\Delta \t{A}+A^3-\t{A}^3.\]
The leading order terms on the right hand side will cause serious technical issues. First of all, it won't make sense if we write
\[ J\Delta A-\t{J}\t\Delta \t{A}=J(\Delta A-\t\Delta \t{A})+(J-\t{J})\t\Delta \t{A}, \]
because the complex structure $J$ and $\t{J}$ are only defined on corresponding normal bundles. Even if we insist to do so by extending the definition of $J$ and $\t{J}$, some first order terms will emerge from the term $\Delta A-\t\Delta \t{A}$, which seems hard to control. Moreover, it is useless to add higher order terms into the energy $\L$, say $\norm{\n A-\tn\t{A}}_2$, because yet an even higher order term will appear and the issue remains unsolved.

From the above discussion, one can see that the difficulty of proving uniqueness is caused by the higher order terms which appear while we take the difference of tensors defined on two different submanifolds. Our solution is to find a more geometric way to measure the difference of two submanifolds. In particular, we will define bundle isomorphisms between tensor bundles on difference submanifolds, to eliminate the extra terms appearing during the substraction. Our key idea comes from the observation that the Gauss map of the SMCF satisfies a \sch map flow~\cite{Song-Gauss-2017} and our investigation on uniqueness of \sch map flows~\cite{SW-uniqueness-2018}.

Given two solutions $F,\t{F}$ of the SMCF, there are two Gauss maps $\rho,\t\rho$ which satisfies a \sch map flow (with time-dependent metrics on the domain manifold) in the Grassmannian $G$. Then following \cite{SW-uniqueness-2018}, we define the intrinsic distance of $d\rho$ and $d\t\rho$ by the parallel transport $\P$ on the Grassmannian manifold. But $d\rho$ and $d\t\rho$ can be identified with the second fundamental forms $A,\t{A}$ respectively. Therefore, we are led to define the intrinsic distance of $A$ and $\t{A}$ by
\[  d(A,\t{A})=d(d\rho, d\t\rho)=|d\rho-\P(d\t\rho)|.\]
In other words, the parallel transport map $\P$ induces a bundle isomorphism between $\rho^*TG=\N\otimes T^*\Si$ and $\t\rho^*TG=\t\N\otimes T^*\t\Si$, such that $J\circ \P=\P\circ \t{J}$. This solves the issue brought by the extrinsic method above. Namely, now we have
\[ J\Delta A-\P(\t{J}\t\Delta\t{A})=J(\Delta A-\P(\t\Delta\t{A})).\]

Actually in our proof, we go one step further and define a bundle isometry $\R_s:\t\N\otimes (T^*\t\Si)^s\to \N\otimes (T^*\Si)^s$ for any $s\in \mathbb{N}$. Then we consider the energy
\[ \L_1(F,\t{F})=\norm{d(\rho,\t\rho)}^2+\norm{A-\R_2(\t{A})}^2+\norm{\n A-\R_3(\tn\t{A})}^2.\]
It turns out that using $\R_s$ instead of $\P$ can help us further reduce the requirements on the regularity of the solutions.

On the other hand, we also need to estimate the difference of the two Laplacian operatiors $\Delta$ and $\t\Delta$. Since the underlying metrics $g$ and $\t{g}$ of the two solutions are different,  we shall encode the difference of the metrics and connections into the total energy by letting
\[ \L_2(F,\t{F})=\norm{g-\t{g}}^2+\norm{\Ga-\t{\Ga}}^2.\]

Moreover, because $\L_1$ is defined via the parallel transport while $\L_2$ is defined by directly identifying $T\Si$ with $T\t\Si$, we also need to fill the gap between $\L_1$ and $\L_2$. This is accomplished by adding a term
\[ \L_3(F,\t{F})=\norm{I-\Q}^2, \]
where $\Q: T\t\Si\to T\Si$ is the isometry induced by $\P$ and $I: T\t\Si\to T\Si$ is the identity map.

Finally, we define the total energy functional by $\L=\L_1+\L_2+\L_3$. Then the uniqueness of SMCF follows from a Gronwall type inequality of $\L$.

\begin{rem}
In \cite{CY-Uniqueness-2007} and \cite{LM-uniqueness-2017}, the authors proved the uniqueness of the Mean Curvature Flow(MCF) by using the parallel transport $\P_{\b{M}}$ in the ambient space $\bar{M}$. It should by pointed out that our parallel transport $\P_G$ in the Grassmannian manifold is more sophisticated than $\P_{\b{M}}$. Indeed, in our setting where $\b{M}=\Real^{n+2}$ is the Euclidean space, the parallel transport $P_{\b{M}}$ is simply the identity map, while $\P_G$ is highly non-trivial. More precisely, $P_{\b{M}}$ provides an isomorphism between the pull-back bundles $F^*T\b{M}$ and $\t{F}^*T\b{M}$, but $\P_G$ establishes an isomorphism between both the normal bundles $\N, \t\N$ and the tangent bundles $T\Si, T\t\Si$.
\end{rem}

\subsection{Parallel transport and distance of tensors }\label{s:intrinsic-distance}

\subsubsection{Intrinsic distance of vectors}

Suppose $M$ is a Riemannian manifold. Let $x,y\in M$ be two distinct points and $X\in T_xM$, $Y\in T_yM$ be two tangent vectors. How can we define the distance between $X$ and $Y$?

Note that in principal $X$ and $Y$ are vectors in two difference tangent spaces, so it does not make sense to do subtraction directly. However, if we embed $M$ into an ambient Euclidean space by $\iota:M\to \Real^K$, then we can think of $X$ and $Y$ as vectors in $\Real^k$ and define
\[ d_1(X,Y):=|d\iota_x(X)-d\iota_y(Y)|_{\Real^K}.\]
We shall call $d_1$ the extrinsic distance of tangent vectors. The disadvantage of this distance is obvious: it relies on the embedding $\iota$ and there is no canonical embedding.

So we seek for more intrinsic distances. One natural way is to use the parallel transport to define an isomorphism between different tangent spaces. More precisely, suppose there is a geodesic $\ga:[0,1]\to M$ connecting $x$ and $y$, then we can parallel transport a vector along $\ga$, which give rise to an isomorphism $\P:T_yM\to T_xM$. Next we define
\[ d_2(X,Y):=|X-\P(Y)|_x.\]
This distance is intrinsic but in general depends on the choice of geodesic $\ga$. Fortunately, if $x$ and $y$ lies close enough, then we have a unique minimizing geodesic and the distance $d_2$ is canonically defined.

The intrinsic distance $d_2$ is particularly useful in proving uniqueness. Moreover, $d_2$ naturally arises when we take derivatives of the distance function on the manifold. Indeed, if $d:M\times M\to \Real^1$ is the distance function on $M$, then
\[ \n_{(X,Y)} d(x,y)=\<-\ga'(0), X\>_x+\<\ga'(1), Y\>_y=\<-\ga'(0), X-\P(Y)\>_x.\]

In fact, the intrinsic distance $d_2$ defined by the parallel transport can be generalized to any vector bundle. Suppose $\pi:(E,\n, h)\to M$ is a vector bundle over $M$. $X\in E_x, Y\in E_y$ are two vectors at $x,y\in M$ which are connected by a geodesic $\ga$. Thus we can find a lift $V:[0,1]\to E$ of $\ga$ such that $\n_sV=0$ and $V(1)=Y$. The parallel transport $\P:E_y\to E_x$ is then obtained by assigning $\P(Y)=V(0)$, and the distance is defined by
\[ d(X,Y)=|X-\P(Y)|_x.\]

\begin{rem}
There is still another intrinsic way to define the distance by using Jacobi fields. Namely, suppose there is a Jacobi fields $V$ on the geodesic $\ga$ such that $V(0)=X$ and $V(1)=Y$, then we define
\[ d_3(X,Y):=\left(\int_0^1|\n_sV|^2ds\right)^{\frac12}.\]
It turns out that this distance is equivalent to $d_2$ defined above, see for example \cite{CJW-Dirac-2013}.
\end{rem}

\begin{rem}
In \cite{DW-Schrodinger-1998}, Ding-Wang proved the uniqueness of \sch map flow for $C^3$ solutions using the extrinsic distance $d_1$. In \cite{McGahagan-2007}, McGahagan showed that the uniqueness holds in a larger function space by using the intrinsic distance $d_2$. This result was improved by Song-Wang~\cite{SW-uniqueness-2018} using the same idea but more intrinsic methods.
\end{rem}

\subsubsection{Parallel transport on Grassmannian manifolds and intrinsic distance of second fundamental forms}

A key idea of the current paper is to define an intrinsic distance of tensors by parallel transport in vector bundles over Grassmannian manifolds, which enable us to compare the second fundamental forms of two submanifolds in a more geometric way. For a preliminary introduction to the geometry of Grassmannian manifolds, we refer to \cite{Song-Gauss-2017}.

First recall that there are two canonical bundles on the Grassmannian manifold $G$, namely the tautological bundle $\G^\top$ where the fiber $\G^\top_\xi$ at each $\xi\in G$ is $\xi$ itself, and the normal counterpart $\G^\bot$ where the fiber $\G^\bot_\xi$ at $\xi$ is its normal complement space $\xi^\bot$.  Moreover, the tangent bundle $\G:=TG$ is isomorphic to the product bundle $\G^\top\otimes \G^\bot$.

Now suppose $F, \t{F}:\Si_0\to \Real^K$ are two submanifolds, and $\rho, \t\rho:\Si_0\to G$ are their Gauss maps. For each $x\in \Si_0$, let $\ga_x:[0,1]\to G$ be a geodesic in $G$ connecting $\rho(x)$ and $\t\rho(x)$. Then by parallel transport in the bundles $\G^\top, \G^\bot, \G$, we have three isomorphisms
\[ \P^\top_x:\G^\top_{\t\rho(x)}\to \G^\top_{\rho(x)}, \P^\bot_x:\G^\bot_{\t\rho(x)}\to\G^\bot_{\rho(x)},
\P_x:\G_{\t\rho(x)}\to \G_{\rho(x)}.\]
In particular, $\P_x=\P^\top_x\otimes\P^\bot_x$.

Next we let $x$ vary on $\Si$, and suppose there is a family of geodesics $\ga_x$ connection $\rho(x)$ and $\t\rho(x)$ for every $x\in \Si$. Then we get three bundle isometries
\[ \P^\top:\t\rho^*\G^\top\to \rho^*\G^\top, \P^\bot:\t\rho^*\G^\bot\to\rho^*\G^\bot,\P:\t\rho^*\G\to \rho^*\G.\]
Again, we have $\P=\P^\top\otimes\P^\bot$.

Recall by definition, $\G^\top_{\rho(x)}=\rho(x)=T_{F(x)}M$ is the tangent plane of $M=F(\Si_0)$ at $F(x)$, and $\G^\bot_{\rho(x)}=\rho(x)^\bot=N_{F(x)}M$ is the normal plane at $F(x)$. Thus we can identify the bundles
\[\H:=F^*TM\cong\rho^*\G^\top,\N:=F^*NM\cong\rho^*\G^\bot.\]
Similarly,
\[ \t\H:=\t{F}^*T\t{M}\cong\t\rho^*\G^\top,\t\N:=\t{F}^*N\t{M}\cong\t\rho^*\G^\bot.\]

 Therefore, the above bundle isomorphisms constructed by the parallel transport actually give us two bundle isometries between the tangent bundles and normal bundles, respectively, i.e.
\[ \P^\top:\t{\H}\to \H, \P^\bot:\t{\N}\to \N. \]

Finally, to extend the parallel transport to tensors including the second fundamental form $A\in\Ga(\N\otimes T^*\Si\otimes T^*\Si)$, we only need to find a bundle isometry between $T^*\Si$ and $\H$. This can be easily accomplished as follows.

First note that the tangent map $dF$ gives a bundle isometry between $T\Si:=(T{\Si_0},g)$ and $\H$. Similarly, $d\t{F}$ gives a bundle isometry between $T\t\Si:=(T{\Si_0},\t{g})$ and $\t\H$. This in turn gives an isometry $\Q$ between the tangent bundles $T\Si$ and $T\t\Si$ by letting the following diagram commute:
\[
\xymatrix{
  T\t\Si  \ar[d]_{\Q}\ar[r]^{d\t{F}}
                & \t{\mathcal{H}} \ar[d]^{\P^\top} \\
  T\Si \ar[r]_{dF}
                & \mathcal{H}          }
\]
In other words, we define
\[Q:=dF^{-1}\circ\P^\top\circ d\t{F}:T\t\Si\to T\Si\]
Since $d\t{F}, \P^\top,dF$ are all isometries, $Q$ is obviously an isometry.

Next we may identify the cotangent bundle $T^*\Si$ (or $T^*\t\Si$) with the tangent bundle $T\Si$ (respectively $T\t\Si$) by sending an orthonormal frame $\{\ep_1^*,\ep_2^*,\cdots, \ep_n^*\}$ to its dual $\{\ep_1,\ep_2,\cdots, \ep_n\}$. Thus by using the above isometry $\Q$, we have a dual isometry on the cotangent bundles $\Q^*:T^*\t{\Si}\to T^*\Si$.

Now for any $s\in \mathbb{N}$, we can define a bundle isometry by
\[ \R_s:=\P^\bot\otimes(\Q^*)^s:\t\N\otimes(T^*\t\Si)^s\to \N\otimes(T^*\Si)^s.\]
which gives an intrinsic distance of tensors $T\in \Ga(\N\otimes(T^*\Si)^s), \t{T}\in \Ga(\t\N\otimes(T^*\t\Si)^s)$ by
\[ d(T,\t{T})=|T-\R_s(\t{T})|_g.\]
In particular, we can define the ``intrinsic distance" of the second fundamental forms by
\[ d(A,\t{A})=|A-\R_2(\t{A})|_g.\]

\subsection{Bundle isomorphisms between two solutions}\label{s:parallel-transport}

\subsubsection{Construction of Bundle isomorphisms}

Now suppose $F$ and $\t{F}$ are two solutions to the SMCF (\ref{e:SMCF}) with same initial data $F_0$. Then there Gauss maps $\rho$ and $\t\rho$ are two solutions to (\ref{e:Gauss}) with same initial data $\rho_0$. We assume that their is a time $T>0$ such that $\rho$ and $\t\rho$ lies sufficiently close to each other such that for any $(t,x)\in [0,T]\times \Si$, there is a unique geodesic $\ga_{(t,x)}:[0,1]\to G$ connecting $\rho(t,x)$ and $\t\rho(t,x)$. The family of geodesics $\ga_{(t,x)}$ gives rise to a map
\[
\begin{aligned}&U\\ &\quad\end{aligned}
\left|
\begin{aligned}
&[0,1]\times[0,T]\times \Si&\to &\ G\\
&\quad\quad (s,t,x) &\mapsto &\ \ga_{(t,x)}(s)
\end{aligned}\right.
\]
By definition, we have $U(0,t,x)=\rho(t,x)$ and $U(1,t,x)=\t\rho(t,x)$. Note that the family of geodesics depends smoothly on its end points given by $\rho$ and $\t\rho$.

As in \cite{SW-uniqueness-2018}, we have the following bound for the derivatives of $U$.
\begin{lem}\label{l:estimate-U}
There holds
\begin{align*}
  |\n_t U|&\le C(|\p_t\rho|+|\p_t\t\rho|)\le C(|\n^2 \rho|+|\tn^2\t\rho|),\\
  |\n U|&\le C(|\n \rho|+|\tn\t\rho|),\\
  |\n^2 U|&\le C(|\n^2 \rho|+|\tn^2\t\rho|)+C(|\n \rho|+|\tn\t\rho|)^2.
\end{align*}
Moreover,
\[ \(\int_0^1|\n\p_s U|^2ds\)^\frac{1}{2}\le |\n \rho-\P(\tn\t\rho)|+C(|\n \rho|+|\tn\t\rho|)d.\]
\end{lem}

From previous discussion in Section~\ref{s:intrinsic-distance}, we know that by parallel transport along $U$, we have bundle isomorphisms
\[ \P^\top:\t\rho^*\G^\top\to \rho^*\G^\top, \quad \P^\perp:\t\rho^*\G^\perp\to \rho^*\G^\perp,
\P=\P^\top\otimes\P^\bot:\t\rho^*\G\to \rho^*\G.\]
In particular, $\P^\top$ induces bundle isomorphisms between tangent bundles and cotangent bundles
\[ \Q=dF^{-1}\circ \P^\top\circ d\t{F}:T\t\Si\to T\Si, ~~~~\Q^*:T^*\t\Si\to T^*\Si.\]
Moreover, we have an extended bundle isomorphism for tensor bundles
\begin{equation}\label{e:Rs}
  \R_s:=\P^\bot\otimes(\Q^*)^s:\t\N\otimes(T^*\t\Si)^s\to \N\otimes(T^*\Si)^s.
\end{equation}

\subsubsection{Estimate of derivatives of $\P$}

By construction, the bundle isomorphism $\P$ preserves the bundle metric. Moreover, since the complex structure on $G$ is parallel, $\P$ also preserves the complex structure of the bundle, i.e. $\P\circ \t{J}=J\circ \P$. However, the two pull-back connections $\n=\rho^*\n^\G$ and $\t\n=\t\rho^*\n^\G$ does not commute with $\P$.  Actually, there holds
\[\bn\P=\n \circ \P - \P\circ \tn.\]
Here we regard $\P$ as a tensor in $\Ga(\t\G^*\otimes \G)=\Hom(\t\G,\G)$ and $\bn$ denotes the induced connection.

To estimate the derivatives of $\P$, we will use the method of moving fames to express $\P$ more explicitly. An alternative method is to use the Jacobi equation, which is applied in \cite{LM-uniqueness-2017}.

For fixed $(t,x)\in I\times\Si$, we first choose an orthonormal frame $\{E_{a}\}$ of $\rho^*\G$ near $\rho(t,x)$, and then parallel transport it along the connecting geodesics to get a local frame $\{\bar{E}_{a}(s)\}$ of $U^*\G$ over $\ga_{(t,x)}(s)$. Denote the resulting frame at $\t\rho(t,x)$ by $\{\t{E}_{a}:=\bar{E}_{a}(1)\}$, which is also orthonormal. We shall call the chosen parallel orthonormal frame on the bundles by \emph{relative frame}.

Obviously, in this frame, we have $\P(\t{E}_{a})=E_{a}$. Thus by letting $\t\Om^a:=\t{E}_a^*$ be the dual frame, we can simply write
\[\P=\t\Om^a\otimes E_a\in \Ga(\t\G^*\otimes \G).\]

Now in the relative frame, the connections has the form $\n=d+A$ and $\tn=d+\t{A}$, where $A=A_{a,i}^bdx^i$ and $\t{A}=\t{A}_{a,i}^bdx^i$ are connection 1-forms. Then we compute
\begin{align*}
  \b{\n}P&=\tn\t\Om^a\otimes E_a+\t\Om^a\otimes\n E_a\\
  &=(-\t{A}^a_b\t\Om^b)\otimes E_a+\t\Om^a\otimes (A^b_a E_b)\\
  &=(A^b_a-\t{A}^b_a)\t\Om^a\otimes E_b.
\end{align*}

Let $\b{\n}=U^*\n^\G=d+\b{A}$ be the pull-back connection on $U^*\G$ over $[0,1]\times[0,T]\times \Si$, where $\b{A}$ is the corresponding connection 1-form in the parallel frame $\{\b{E}_a\}$. Since the frame we choose are parallel along $s$-direction, we can write
\[ \b{A}=\b{A}_tdt+\b{A}_idx^i.\]
The curvature form of $\bn$ is given by
\[ \b{F}=d\b{A}+[\b{A},\b{A}]=U^*R^\G.\]
It follows that the $ds\wedge dx^i$ component of $\b{F}$ is
\[\b{F}_{si}=\p_s\b{A}_i=R^\G(\p_sU, \n_i U).\]

On the other hand,since $\n$ and $\tn$ are the restrictions of $\bn$ on $\rho^*\G=U^*\G|_{s=0}$ and $\t\rho^*\G=U^*\G|_{s=1}$, we have $A(t,x)=\b{A}(0,t,x)$ and $\t{A}(t,x)=\b{A}(1,t,x)$. Hence the difference of two connections is given by
\[  A_i-\t{A}_i=\b{A}_i|_{s=0}-\b{A}_i|_{s=1}=-\int_0^1\p_s\b{A_i}ds=-\int_0^1\b{F}_{is}ds. \]
Therefore, we get
\[
  \bn\P=-\int_0^1R^\G(\p_sU, \n_i U)_a^bds\cdot dx^i\otimes\t\Om^a\otimes E_b.
\]
Taking another derivative, we get
\begin{align*}
  \bn^2\P&=-\bn\(\int_0^1R^\G(\p_sU, \n_i U)_a^bds\cdot dx^i\otimes\t\Om^a\otimes E_b\)\\
  &=-\int_0^1\bn(R^\G(\p_sU, \n_i U)_a^b)ds\cdot dx^i\otimes\t\Om^a\otimes E_b+\int_0^1R^\G(\p_sU, \n U)ds\cdot \bn(dx^i\otimes\t\Om^a\otimes E_b)\\
  &=-\int_0^1(\n^GR^\G)(\p_sU, \n U,\n U)ds-\int_0^1 R^\G(\n\p_sU, \n U)ds-\int_0^1 R^\G(\p_sU, \n^2 U)ds\\
  &\quad +\int_0^1R^\G(\p_sU, \n U)ds\cdot (A-\t{A}).
\end{align*}

As a conclusion, we have
\begin{lem}\label{l:estimate-P}
  The derivatives of $\P$ satisfies
\begin{align*}
|\bn_t\P|&\le Cd|\n_tU|\\
  |\bn \P|&\le Cd|\n U|.\\
|\tn^2\P|&\le C(|\n U|+|\n^2U|)d + C|\n U||\n\p_s U|,
\end{align*}
where the constant $C$ only depends on $G$.
\end{lem}

\begin{rem}
  By replacing the bundle $\G$ with $\G^\top$ and $\G^\bot$ in the above arguments, it is easy to see that $\P^\top$ and $\P^\bot$ satisfies same estimates. Moreover, since $\bn\Q=\bn\P^{\top}$, $\Q$ also satisfies same estimates. It follows that for any $s\in \mathbb{N}$, the parallel translation $\R_s$ satisfies same estimates.
\end{rem}

\subsection{Difference of Laplacian operators}

\subsubsection{Difference of Laplacian with different metrics}

Let $(\E,\n,h)$ be a vector bundle over a manifold $\Si$, suppose there are two difference metric $g$ and $\t{g}$ on $\Si$. We denote by $\n_g, \Delta_g$ and $\n_{\t{g}}, \Delta_{\t{g}}$ the corresponding induced covariant derivatives and (rough) Laplacians. Then for any section $\Phi\in \Ga(E)$, by definition
\[\begin{aligned}
\Delta_g \Phi - \Delta_{\t{g}} \Phi
&=(g^{ij}-\t{g}^{ij})\n_i\n_j\Phi-(g^{ij}\Ga_{ij}^k-\t{g}^{ij}\t{\Ga}_{ij}^k)\n_k\Phi\\
&=(g^{ij}-\t{g}^{ij})\n_{g,ij}^2\Phi-\t{g}^{ij}(\Ga_{ij}^k-\t{\Ga}_{ij}^k)\n_k\Phi.
\end{aligned}\]
Thus
\[ |\Delta_g \Phi - \Delta_{\t{g}} \Phi|_h\le |g^{-1}-\t{g}^{-1}|_g|\n_g^2\Phi|_{h\oplus g}
+|\t{g}^{-1}|_g|\Ga-\t{\Ga}|_g|\n_g \Phi|_{h\oplus g},\]
where we use the metric $g$ for tensors on the right hand side.

\subsubsection{Difference of Laplacian with different connections}

Given two vector bundles $(E,\n, h)$ and $(\t{E},\tn, \t{h})$ over a Riemannian manifold $(\Si, g)$. Suppose there is a bundle isomorphism $\P:\t{E}\to E$ which preserves the metric, then we have for any $\Phi\in \Ga(\t{E})$,
\[\n(\P\Phi)=(\b\n\P)\Phi+\P(\tn\Phi),\]
where $\bn$ denotes the induced connection on $\t{E}^*\otimes E$. It follows
\begin{align*}
  \n^2(\P\Phi)&=\n((\bn \P)\Phi+\P(\tn\Phi)\\
  &=(\bn^2\P)\Phi+2(\bn \P)\tn\Phi+\P(\tn^2\Phi).
\end{align*}
Taking trace, we get
\[ \Delta(\P\Phi)-\P(\t{\Delta}\Phi)=(\b{\Delta}\P)\Phi+2\bn\P\cdot\tn\Phi.\]
Therefore,
\[ |\Delta(\P\Phi)-\P(\t{\Delta}\Phi)|_h\le |\b{\Delta}\P|_{\b{h}}|\Phi|_{\t{h}}+2|\bn\P|_{\b{h}}|\tn\Phi|_{\t{h}}, \]
where $\t{h}=\t{h}\oplus h$ is the induced metric on $\t{E}^*\otimes E$.

\subsubsection{Difference of Laplacian with different metrics and connections}

Now for two vector bundles $\pi:(E,\n,h)\to (\Si_0,g)$ and $\t\pi:(\t{E},\tn,\t{h})\to (\Si_0,\t{g})$, suppose there is a bundle isomorphism $\P:\t{E}\to E$ which preserves the metrics on the bundle, then we have for any $\Phi\in \Ga(\t{E})$,
\begin{align*}
  |\Delta_g(\P\Phi)-\P(\t{\Delta}_{\t{g}}\Phi)|_{h}
  &\le |\Delta_g(\P\Phi)-\P(\t{\Delta}_g\Phi)|_{h}+|\P(\t{\Delta}_g\Phi)-\P(\t{\Delta}_{\t{g}}\Phi)|_{h}\\
  &= |\Delta_g(\P\Phi)-\P(\t{\Delta}_g\Phi)|_{h}+|\t{\Delta}_g\Phi-\t{\Delta}_{\t{g}}\Phi|_{\t{h}}
\end{align*}
By the estimates above, we get
\begin{lem}\label{l:estimate-Laplacian}
Under above settings, we have
\begin{align*}
  &|\Delta_g(\P\Phi)-\P(\t{\Delta}_{\t{g}}\Phi)|_{h}\\
  &\le 2|\bn\P|_{\b{h}}|\tn_{\t{g}} \Phi|_{\t{h}}+|\b{\Delta}\P|_{\b{h}}|\Phi|_{\t{h}}+|g^{-1}-\t{g}^{-1}|_{\t{g}}|\tn^2_{\t{g}}\Phi|_{\t{h}\oplus\t{g}}
  +|g^{-1}|_{\t{g}}|\Ga-\t\Ga|_{\t{g}}|\tn_{\t{g}}\Phi|_{\t{h}\oplus\t{g}}.
\end{align*}
\end{lem}

\

As a conclusion of Lemma~\ref{l:estimate-Laplacian}, Lemma~\ref{l:estimate-P} and Lemma~\ref{l:estimate-U}, we obtain
\begin{cor}\label{l:corollary-laplacian}
For any tensor $\Phi\in \Ga(\t\N\otimes (T^*\t\Si)^s)$ and the parallel transport $\R_s$ defined by (\ref{e:Rs}), if $g$ and $\t{g}$ are equivalent, then there holds
\[
  |\Delta_g(\R_s\Phi)-\R_s(\t{\Delta}_{\t{g}}\Phi)|
\le C(d+|A-\P(\t{A})|+|g-\t{g}|+|\Ga-\t\Ga|).
\]
where $C$ depends on $\norm{\rho}_{2,\infty}, \norm{\t\rho}_{2,\infty}$ and $\norm{\Phi}_{2,\infty}$.
\end{cor}

In the above theorem, since the metrics $g$ and $\t{g}$ are assumed to be equivalent, i.e. there exists constants such that $C_1g\le \t{g}\le C_2g$, we can use either one for tensors. Thus we omit the subscripts denoting the metrics here and subsequently.

\subsection{Estimate of $\L$}\label{s:estimate-L}

By the discussions in Section~\ref{s:parallel-transport} above, we have the parallel transport $\R_s$ defined by (\ref{e:Rs}). Now we define the energy functional $\L(F,\t{F})=\L_1+\L_2+\L_3$ where
\[ \L_1=\norm{d(\rho,\t\rho)}_2^2+\norm{A-\R_2(\t{A})}_2^2+\norm{\n A-\R_3(\tn\t{A})}_2^2,\]
and
\[ \L_2=\norm{g-\t{g}}_2^2+\norm{\Ga-\t{\Ga}}_2^2,\]
and
\[ \L_3=\norm{I-\Q}_2^2.\]
Here we assume the metrics $g$ and $\t{g}$ are equivalent, thus we may choose a fixed background metric $g_0$ (e.g. the induced metric of the initial immersion $F_0$) to define the norm in $\L_2$. In the definition of $\L_3$, $I$ is the identity map from $T\t\Si$ to $T\Si$ and the norm is taken with regard to the induced metric $\b{g}=\t{g}\oplus g$.

In what follows, we will derive a Gronwall type estimate of $\L$.

\subsubsection{Estimate of $\L_1$}

First using equation (\ref{e:Gauss}), we compute
\begin{align*}
  \frac12\frac{d}{dt}\int d(\rho,\t\rho)^2&=\int \<-\ga'(0), \p_t\rho\>+\<\ga'(1),\p_t\t\rho\>\\
  &=\int \<-\ga'(0),\J(\rho)\Delta\rho-\P(\J(\t\rho)\t{\Delta}\t\rho)\>\\
  &=\int \<-\ga'(0),\J(\rho)\tr_g\n d\rho - \J(\rho)\P(\tr_{\t{g}}\tn d\t\rho)\>\\
  &=\int \<\J(\rho)\ga'(0),\tr_g(\n A-\R_3(\tn \t{A}))\>\\
  &\le \int d^2+\int|\n A-\R_3(\tn \t{A})|^2,
\end{align*}
where $d:=d(\rho,\t\rho)$.

For the second term of $\L_1$, recall that by Lemma~\ref{l:evolution-equ},
\[ \n_tA=J\Delta A+A^3, \]
and
\[ \tn_t\t{A}=\t{J}\t{\Delta}\t{A}+\t{A}^3.\]
It follows
\begin{align*}
  \n_tA-\n_t\R(A)&=J\Delta A+A^3-\R(\t{J}\t{\Delta}\t{A}+\t{A}^3)-\bn_t\R(\t{A})\\
  &=J(\Delta A- \R(\t{\Delta}\t{A}))+(A^3-\R(\t{A})^3)-\bn_t\R(\t{A})\\
  &=J\Delta(A-\R(\t{A}))+J(\Delta (\R\t{A})- \R(\t{\Delta}\t{A}))+(A^3-\R(\t{A})^3)-\bn_t\R(\t{A}).
\end{align*}
Hence
\begin{align*}
  \frac12\frac{d}{dt}\int |A-\R(\t{A})|_g^2&=\int \<A-\R(\t{A}),\n_tA-\n_t\R(\t{A})\>\\
  &=\int \<A-R(\t{A}),J(\Delta (\R\t{A})- \R(\t{\Delta}\t{A}))+(A^3-\R(\t{A})^3)-\bn_t\R(\t{A})\>\\
  &\le C\(\int |A-\R(\t{A})|_g^2+\int|\Delta (\R\t{A})- \R(\t{\Delta}\t{A})|^2+\int|\bn_t\R|^2\)
\end{align*}
where $C$ depends on $\norm{A}_{\infty}, $ and $\norm{\t{A}}_{\infty}$.

By Lemma~\ref{l:estimate-P} and Lemma~\ref{l:estimate-U},
\[ |\bn_t\R|\le Cd|\p_t U|\le C(|\n A|+|\tn\t{A}|)d.\]
On the other hand, by Corollary~\ref{l:corollary-laplacian},
\[ |\Delta (\R\t{A})- \R(\t{\Delta}\t{A})|\le C(d+|A-\P(\t{A})|+|g-\t{g}|+|\Ga-\t\Ga|).\]
where $C$ depends on $\norm{\rho}_{2,\infty}, \norm{\t\rho}_{2,\infty}$ and $\norm{\t{A}}_{2,\infty}$.
Therefore, we arrive at
\[ \frac12\frac{d}{dt}\int |A-\R(\t{A})|^2\le C\(\int d^2+\int |A-\R(\t{A})|_g^2+\int|g-\t{g}|^2+\int|\Ga-\t\Ga|^2\).\]
where $C$ depends on $\norm{A}_{1,\infty}$ and $\norm{\t{A}}_{2,\infty}$.

For the third term of $\L_1$, recall that by Lemma~\ref{l:evolution-equ},
\[ \n_t\n A=J\Delta \n A + A\#A\#\n A,\]
and
\[ \tn_t\tn \t{A}=\t{J}\t\Delta\tn\t{A} + \t{A}\#\t{A}\#\tn\t{A}.\]
It follows that
\begin{align*}
  \n_t\n A-\n_t(\R\tn\t{A})&=J\Delta \n A + A\#A\#\n A-\R(\t{J}\t\Delta\tn\t{A} - \t{A}\#\t{A}\#\tn\t{A})-\bn_t\R(\tn\t{A})\\
  &=J\Delta(\n A-\R(\tn\t{A}))+J(\Delta(\R(\tn\t{A}))-\R(\t\Delta\tn\t{A}))\\
  &\ \ \ \ \ \ \ +(A\#A\#\n A-\R(\t{A})\#\R(\t{A})\#\R(\tn\t{A}))-\bn_t\R(\tn\t{A})
\end{align*}
So we have
\begin{align*}
  &\quad\frac12\frac{d}{dt}\int|\n A-\R(\tn\t{A})|^2\\
  &=\int\<\n A-\R(\tn\t{A}),\n_t\n A-\n_t(\R\tn\t{A})\>\\
  &=\int\<\n A-\R(\tn\t{A}),J(\Delta(\R(\tn\t{A}))-\R(\t\Delta\tn\t{A}))\>\\
  &\ \ \ \ \ \ \ +\int\<\n A-\R(\tn\t{A}),(A\#A\#\n A-\R(\t{A})\#\R(\t{A})\#\R(\tn\t{A}))-\bn_t\R(\tn\t{A})\>\\
  &\le C\(\int|A-\R(\t{A})|_g^2+\int |\n A-\R(\tn\t{A})|^2\right.\\
  &\ \ \ \ \ \ \ +\left.\int|\Delta(\R(\tn\t{A}))-\R(\t\Delta\tn\t{A})|^2+\int|\bn_t\R(\tn\t{A})|^2\).
\end{align*}
Again by Corollary~\ref{l:corollary-laplacian}, we have
\[
   |\Delta (\R(\tn\t{A}))- \R(\t{\Delta}\tn\t{A})|
   \le C(d+|A-\P(\t{A})|+|g-\t{g}|+|\Ga-\t\Ga|).
\]
where $C$ depends on $\norm{\rho}_{2,\infty}, \norm{\t\rho}_{2,\infty}$ and $\norm{\tn\t{A}}_{2,\infty}$.
 Therefore, we obtain
 \begin{align*}
   &\frac12\frac{d}{dt}\int|\n A-\R(\tn\t{A})|^2\\
   &\le C\(\int d^2+\int|A-\R_2(\t{A})|_g^2+\int |\n A-\R_3(\tn\t{A})|^2+\int|g-\t{g}|^2+\int|\Ga-\t\Ga|^2\).
 \end{align*}

 In conclusion, we have
 \begin{equation}\label{e:L1}
 \frac{d}{dt}\L_1\le C(\L_1+\L_2).
 \end{equation}
 where $C$ depends on $\norm{A}_{2,\infty}$ and $\norm{\t{A}}_{3,\infty}$.

\subsubsection{Estimate of $\L_2$}

First recall that by Lemma~\ref{l:evolution-equ},
\[ \p_t g=-2\<JH,A\>, \quad \p_t \t{g}=-2\<\t{J}\t{H},\t{A}\>.\]
Thus
\begin{align*}
  \p_tg-\p_t\t{g}&=-2\<JH,A\>+2\<\P^\bot(\t{J}\t{H}),\P^\bot(\t{A})\>\\
  &=-2\<J(H-\P^\bot(\t{H})), A\>-2\<\P^\bot(\t{J}\t{H}),A-\P^\bot(\t{A})\>
\end{align*}
It follows
\[ |\p_t(g-\t{g})|\le C|A-\P^\bot(\t{A})|\le C(|A-\R_2(\t{A})|+|I-\Q||\t{A}|).\]
and
\[ \frac12\p_t\int|g-\t{g}|^2\le C\(\int|g-\t{g}|^2+\int|A-\R_2(\t{A})|^2+\int|I-\Q|^2\),\]
where $C$ depends on $\norm{A}_\infty$ and $\norm{\t{A}}_\infty$.

Next, recall that the evolution equation of the Christoffel symbol is
\[\p_t\Ga=g^{-1}\#\n\p_tg, \quad \p_t\t\Ga=\t{g}^{-1}\#\t{\n}\p_t\t{g}.\]
Moreover,
\[ \n\p_t g=-2\n\<JH,A\>, \quad \tn\p_t \t{g}=-2\tn\<\t{J}\t{H},\t{A}\>.\]
It follows
\[\begin{aligned}
|\p_t(\Ga-\t\Ga)|&=|g^{-1}\#\n\p_tg-\t{g}^{-1}\#\t{\n}\p_t\t{g}|\\
&\le C(|g^{-1}-\t{g}^{-1}|+|A-\t{A}|+|\n A-\t{\n}\t{A}|)\\
&\le C(|g^{-1}-\t{g}^{-1}|+|A-\R_2(\t{A})|+|\n A-\R_3(\t{\n}\t{A})|+(|\t{A}|+|\tn\t{A}|)|I-\Q|).
\end{aligned}\]
So we have
\begin{align*}
  &\quad\frac12\frac{d}{dt}\int|\Ga-\t\Ga|^2=\int\<\Ga-\t\Ga,\p_t(\Ga-\t\Ga)\>\\
  &\le C\(\int|\Ga-\t\Ga|^2+\int|g-\t{g}|^2+\int|A-\R_2(\t{A})|^2+\int |\n A-\R_3(\tn\t{A})|^2+\int|I-\Q|^2\),
\end{align*}
where $C$ depends on $\norm{A}_{1,\infty}$ and $\norm{\t{A}}_{1,\infty}$.

 In conclution, we have
 \begin{equation}\label{e:L2}
 \frac{d}{dt}\L_2\le C(\L_1+\L_2+\L_3).
 \end{equation}
 where $C$ depends on $\norm{A}_{1,\infty}$ and $\norm{\t{A}}_{1,\infty}$.

\subsubsection{Estimate of $\L_3$}

To estimate the time derivative of $\L_3$, we need to compute the evolution equation of the identity map $I:T\t\Si\to T\Si$. This is most easily done in natural coordinates.

Denote the natural coordinates in $\Si$ and $\t\Si$ by $x^i$ and $\t{x}^i$ (which is actually the same one on the parameter space $\Si_0$). Let $\p_i=\frac{\p}{\p x^i}$ and $\t\p_i=\frac{\p}{\p \t{x}^i}$. In this coordinates, the identity map is simply $I= d\t{x}^i\otimes \p_i$. Then we can compute
\[ \bn_t I=\tn_td\t{x}^i\otimes \p_i+d\t{x}^i\otimes \n_t\p_i=(\Ga_{tj}^i-\t\Ga_{tj}^i)d\t{x}^j\otimes \p_i.\]
On the other hand, we have
\[ \Ga_{tj}^i=\frac{1}{2}g^{ik}\p_tg_{jk}=-g^{ik}\<JH, A_{jk}\>,\]
and
\[ \t\Ga_{tj}^i=\frac{1}{2}\t{g}^{ik}\p_t\t{g}_{jk}=-\t{g}^{ik}\<\t{J}\t{H}, \t{A}_{jk}\>.\]
It follows
 \[ |\bn_t I|\le C(|g-\t{g}|+|A-\R_2(\t{A})|+|I-\Q|),\]
 where $C$ depends on $\norm{A}_{\infty}$ and$\norm{\t{A}}_{\infty}$.

 From this inequality and Lemma~\ref{l:estimate-P}, we can estimate
 \begin{align*}
   \frac12\frac{d}{dt}\int|I-\Q|^2&=\int\<I-\Q, \bn I-\bn \Q\>\\
   &\le \int |I-\Q|^2 +\int |\bn I|^2 +\int |\bn \Q|^2 \\
   &\le C\(\int |I-\Q|^2 + \int|g-\t{g}|^2 + \int |A-\R_2(\t{A})|^2 +\int d^2\).
 \end{align*}

 In conclusion, we get
 \begin{equation}\label{e:L3}
   \frac{d}{dt}\L_3\le C(\L_1+ \L_2+\L_3).
 \end{equation}
 where $C$ depends on $\norm{A}_{\infty}$ and$\norm{\t{A}}_{\infty}$.

\subsection{Proof of uniqueness} \label{s:proof-of-uniqueness}

Now we are ready to prove Theorem~\ref{t:main-uniqueness}.

\begin{proof}[Proof of Theorem~\ref{t:main-uniqueness}]
By assumption $F\in L^\infty([0,T], \S_2)$ and $\t{F}\in L^\infty([0,T], \S_3)$, the norms $\norm{A}_{2,\infty}$ and $\norm{A}_{3,\infty}$ are uniformly bounded for all time $t\in [0,T]$. It is easy to see that the induced metric $g$ and $\t{g}$ are equivalent.

Moreover, the Gauss maps $\rho, \t\rho$ of $F, \t{F}$ satisfy the \sch map flow (\ref{e:Gauss}) with same initial value $\rho_0$ given by the Gauss map of $F_0$. Since $\tau_g(\rho)=\n H$ and $\tau_{\t{g}}(\t\rho)=\tn\t{H}$ are uniformly bounded, there exists a time $T'>0$ such that for all $(x,t)\in \Si_0\times [0,T']$, the image $\rho(x,t), \t\rho(x,t)$ lies in a sufficiently small geodesic ball of the Grassmannian manifold $G$, which is centered at $\rho(x,t)$.
Therefore, we can construct the family of geodesics $U$ connecting $\rho$ and $\t\rho$, hence the parallel transport/bundle isomorphisms $\P, \Q, \R_s$ as in Section~\ref{s:parallel-transport}.

Then we can define the energy function $\L$ as in Section~\ref{s:estimate-L}. Next, combining estimates  (\ref{e:L1}), (\ref{e:L2}) and (\ref{e:L3}), we obtain
\begin{equation*}
  \frac{d}{dt}\L\le C\L, ~~\forall t\in[0,T'],
\end{equation*}
where $C$ depends on $\norm{A}_{2,\infty}$ and$\norm{\t{A}}_{3,\infty}$.
It follows that $\L(t)\le e^{Ct}\L(0)$ for $t\in [0,T']$. But $\L(0)=0$ since $F$ and $\t{F}$ have the same initial value. Hence $\L$ vanishes identically and $F=\t{F}$ on $[0,T']$.

Next, starting from the time $T'$, we can repeat the above arguments to get $F=\t{F}$ on another time interval $[T', T'+T'']$. Note that $T''$ can be chosen to be equal to $T'$, since it only depends on $\norm{A}_{1,\infty}$ and $\norm{\t{A}}_{1,\infty}$ which are uniformly bounded.

Therefore, after repeating the arguments for finitely many times, we can get the uniqueness on the whole time interval $[0,T]$, which finishes the proof of the theorem.
\end{proof}

\section{Appendix}

Let $F:\Si^n\to \Real^m$ be a compact immersed submanifold. In this appendix, we show that the energy $\E_k=\vol+\norm{H}_p^2 +\norm{A}_{H^{k,2}}^2$ is equivalent to the Sobolev norm of the Gauss map $\b{\E}_k=\norm{d\rho}_{W^{k,2}}^2$, where the Sobolev norms will be defined later. For a preliminary introduction to the geometry of Grassmannian manifolds and Gauss maps, we refer to \cite{Song-Gauss-2017}.

Let $D$ denote the usual connection of the exterior product space $\Ld:=\Ld^n\Real^m$, which is induced by the standard derivative on $\Real^m$. Let $\n$ denote the Levi-Civita connection of the Grassmannian manifold $G:=G(n,m-n)$ and $\Pi\in \Ga(T^*G\otimes T^*G\otimes NG)$ denote the second fundamental form of $G$ as a submanifold in $\Ld$. We can regard $\Pi$ as an 1-form $\Pi=\Pi_a dy^a$ on $G$, where each entry $\Pi_a\in \Ga(T^*G\otimes NG)$ is a linear map from $TG$ to $NG$.

The Gauss map of $F$ is a map $\rho:\Si\to G\hookrightarrow \Ld$. We will still denote the pull-back connections on the pull-back bundles $\rho^*T\Ld\otimes (T^*\Si)^s$ and $\rho^*TG\otimes (T^*\Si)^s$ by $D$ and $\n$ respectively. Then applying $D$ on $d\rho\in\Ga(\rho^*TG\otimes T^*\Si)$, we have
\[Dd\rho= (Dd\rho)^\top + (Dd\rho)^\bot=\n d\rho + \rho^*\Pi(d\rho),\]
where $\rho^*\Pi=\Pi_a\p_i\rho^adx^i$ is the pull-back 1-form on $\Si$. Since we can identify $d\rho$ with the second fundamental form $A$ of the immersion $F$, we can write the above equality as
\[ DA=\n A+\Pi(\rho)\#A^2,\]
where $\#$ denote linear combinations. Taking once more derivative, we get
\begin{align*}
  D^2A&=D\n A+D(\Pi(\rho)\#A^2)\\
  &=(\n^2A+\Pi\#A\#\n A)+(D\Pi\#A^3+\Pi\#DA\#A)\\
  &=\n^2A+\Pi\#A\#\n A+(D\Pi+\Pi\#\Pi)\#A^3.
\end{align*}
Inductively, we can derive for any $k\ge1$,
\begin{equation}\label{e-app-1}
  D^kA=\n^kA+\sum_J C_J(\Pi)\#\n^{j_1}A\#\cdots\#\n^{j_s}A,
\end{equation}
or equivalently,
\begin{equation}\label{e-app-2}
  D^{k}d\rho=\n^kA+\sum_J C_J(\Pi)\#\n^{j_1}d\rho\#\cdots\#\n^{j_s}d\rho,
\end{equation}
where $C_J(\Pi)$ is a linear combination of $\Pi$ and its derivatives, and the summation is taken for all indices $J=(j_1,\cdots, j_s)$ with
\begin{equation*}
 0\le j_i\le k-1 \text{~and~} \sum_i(j_i+1)=\sum_i j_i + s = k+1,
\end{equation*}

Now for $k\ge 1$, define the Sobolev norms
\[ \norm{A}_{H^{k,2}}=\(\sum_{l=0}^k\int |\n^l A|^2\)^{1/2},\]
and
\[ \norm{d\rho}_{W^{k,2}}=\(\sum_{l=0}^k\int |D^l d\rho|^2\)^{1/2}.\]

\begin{thm}\label{t-equivalence}
For any compact immersed submanifold $F:\Si^n\to \Real^m$ and $k\ge l_0:=[n/2]$, suppose $\vol+\norm{H}_p\le B$ for some $p>n$, then there exists two constants $C_1(k, B)$ and $C_2(k, B)$ which only depend on $k$ and $B$ (and is independent of the submanifold), such that
\begin{equation}\label{e:app-1}
 \norm{A}_{H^{k,2}}\le C_1(k, B)\sum_{i=1}^{k+1}\norm{d\rho}_{W^{k,2}}^i,
 \end{equation}
and
\begin{equation}\label{e:app-2}
  \norm{d\rho}_{W^{k,2}}\le C_2(k, B)\sum_{i=1}^{k+1}\norm{A}_{H^{k,2}}^i.
\end{equation}
\end{thm}
\begin{proof}
Since by assumption $\vol+\norm{H}_p\le B$, we have uniform interpolation inequalities by Theorem~\ref{t-lp}. Then in view of (\ref{e-app-1}) and (\ref{e-app-2}), the proof follows step by step from Proposition 2.2 in Ding-Wang~\cite{DW-Schrodinger-2001},.
\end{proof}

\begin{thm}\label{c:equivalence}
  For any compact immersed submanifold $F:\Si^n\to \Real^m$ and $k\ge l_0:=[n/2]$, the energy $\E_k=\vol+\norm{H}_p^2 +\norm{A}_{H^{k,2}}^2$ is equivalent to $\bar{\E}_k=\norm{\rho}_{W^{k+1,2}}^2$. Namely, there exists two functions $f_k$ and $g_k$ which only depend on $k$ (and  is independent of the submanifold), such that
  \[ \b{\E}_k\le f_k({\E}_k), \quad {\E}_k\le g_k(\b{\E}_k).\]
\end{thm}
\begin{proof}
  First note that $|\rho|=1$ since the image of $\rho$ lies in $G$ which is contained in the unit sphere in $\Ld$. Thus $\int|\rho|^2=\vol$ and $\norm{\rho}_{W^{k+1,2}}^2=\vol +\norm{d\rho}_{W^{k,2}}^2$.

  If we take $B=\E_k$ in Theorem~\ref{t-equivalence}, then by (\ref{e:app-2}),
  \[ \b{\E}_k=\vol +\norm{d\rho}_{W^{k,2}}^2\le \vol+C_2(k, B)\sum_{i=1}^{k+1}\norm{A}_{H^{k,2}}^i.\]
  So it is easy to find the desired function $f_k$ such that $\b{\E}_k\le f_k(\E_k)$.

  On the other hand, by letting $j=1, k=l_0+1, r=\infty, q=2$ in the universal interpolation inequality of Theorem~\ref{t-universal}, we have
  \[ \norm{A}_p=\norm{d\rho}_p\le C\norm{D^{l_0}d\rho}^{\frac{1}{l_0+1}}_2\norm{\rho}^{1-\frac{1}{l_0+1}}_{\infty}\le C\b{\E}_k^{\frac{1}{l_0+1}}.\]
  where $p=2(l_0+1)>n$. Therefore, we get
  \[\vol+\norm{H}_p^2\le B':=\b{\E}_k+C\b{\E}_k^{\frac{2}{l_0+1}}.\]
  Then by (\ref{e:app-1}) in Theorem~\ref{t-equivalence}, there exists a function $g_k$ such that ${\E}_k\le g_k(\b{\E}_k)$.
\end{proof}

\bibliographystyle{alpha}
\bibliography{SMCF-bib}

\begin{thebibliography}{BIKT11}

\bibitem[Aub98]{Aubin-book-1998}
Thierry Aubin.
\newblock {\em Some nonlinear problems in {R}iemannian geometry}.
\newblock Springer Monographs in Mathematics. Springer-Verlag, Berlin, 1998.

\bibitem[BIKT11]{BIKT-Global-2011}
Ioan Bejenaru, Alexandru~D Ionescu, Carlos~E Kenig, and Daniel Tataru.
\newblock Global {S}chr\"{o}dinger maps in dimensions {$d\geq 2$}: small data
  in the critical {S}obolev spaces.
\newblock {\em Ann. of Math. (2)}, 173(3):1443--1506, 2011.

\bibitem[Can75]{Cantor-Sobolev-1975}
M.~Cantor.
\newblock Sobolev inequalities for {R}iemannian bundles.
\newblock In {\em Differential geometry ({P}roc. {S}ympos. {P}ure {M}ath.,
  {V}ol. {XXVII}, {S}tanford {U}niv., {S}tanford, {C}alif., 1973), {P}art 2},
  pages 171--184. Amer. Math. Soc., Providence, R.I., 1975.

\bibitem[CJW13]{CJW-Dirac-2013}
Qun Chen, J\"{u}rgen Jost, and Guofang Wang.
\newblock The maximum principle and the {D}irichlet problem for
  {D}irac-harmonic maps.
\newblock {\em Calc. Var. Partial Differential Equations}, 47(1-2):87--116,
  2013.

\bibitem[CY07]{CY-Uniqueness-2007}
Bing-Long Chen and Le~Yin.
\newblock Uniqueness and pseudolocality theorems of the mean curvature flow.
\newblock {\em Comm. Anal. Geom.}, 15(3):435--490, 2007.

\bibitem[Din02]{Ding-ICM-2002}
Weiyue Ding.
\newblock On the {S}chr\"{o}dinger flows.
\newblock In {\em Proceedings of the {I}nternational {C}ongress of
  {M}athematicians, {V}ol. {II} ({B}eijing, 2002)}, pages 283--291. Higher Ed.
  Press, Beijing, 2002.

\bibitem[DW98]{DW-Schrodinger-1998}
Weiyue Ding and Youde Wang.
\newblock Schr\"{o}dinger flow of maps into symplectic manifolds.
\newblock {\em Sci. China Ser. A}, 41(7):746--755, 1998.

\bibitem[DW01]{DW-Schrodinger-2001}
Weiyue Ding and Youde Wang.
\newblock Local {S}chr\"{o}dinger flow into {K}\"{a}hler manifolds.
\newblock {\em Sci. China Ser. A}, 44(11):1446--1464, 2001.

\bibitem[Gom04]{Gomez-Thesis-2004}
Hector~Hernando Gomez.
\newblock {\em Binormal motion of curves and surfaces in a manifold}.
\newblock ProQuest LLC, Ann Arbor, MI, 2004.
\newblock Thesis (Ph.D.)--University of Maryland, College Park.

\bibitem[Jer02]{Jerrard-GP-2002}
Robert~L. Jerrard.
\newblock Vortex filament dynamics for {G}ross-{P}itaevsky type equations.
\newblock {\em Ann. Sc. Norm. Super. Pisa Cl. Sci. (5)}, 1(4):733--768, 2002.

\bibitem[JS15]{JS-JEMS-2015}
Robert~L. Jerrard and Didier Smets.
\newblock On the motion of a curve by its binormal curvature.
\newblock {\em J. Eur. Math. Soc.}, 17(6):1487--1515, 2015.

\bibitem[Khe12]{Khesin-SMCF-2012}
Boris Khesin.
\newblock Symplectic structures and dynamics on vortex membranes.
\newblock {\em Mosc. Math. J.}, 12(2):413--434, 461--462, 2012.

\bibitem[KPV04]{KPV-quai-linear-Schrodinger-2004}
Carlos~E. Kenig, Gustavo Ponce, and Luis Vega.
\newblock The {C}auchy problem for quasi-linear {S}chr\"{o}dinger equations.
\newblock {\em Invent. Math.}, 158(2):343--388, 2004.

\bibitem[KY19]{Khesin-Yang-2019}
Boris Khesin and Cheng Yang.
\newblock Higher-dimensional {H}asimoto transform for vortex membranes:
  counterexamples and generalizations.
\newblock {\em arXiv preprint arXiv:1902.08834}, 2019.

\bibitem[Li18]{Li-Global-2018}
Ze~Li.
\newblock Global 2{D} {S}chr\"{o}dinger map flows to {K}\"{a}hler manifolds
  with small energy.
\newblock {\em arXiv preprint arXiv:1811.10924}, 2018.

\bibitem[LM17]{LM-uniqueness-2017}
Man-Chun Lee and John Man-shun Ma.
\newblock Uniqueness theorem for non-compact mean curvature flow with possibly
  unbounded curvatures.
\newblock {\em arXiv preprint arXiv:1709.00253}, 2017.

\bibitem[LP15]{LP-book-2015}
Felipe Linares and Gustavo Ponce.
\newblock {\em Introduction to nonlinear dispersive equations}.
\newblock Universitext. Springer, New York, second edition, 2015.

\bibitem[Man02]{Mantegazza-GAFA-2002}
Carlo Mantegazza.
\newblock Smooth geometric evolutions of hypersurfaces.
\newblock {\em Geom. Funct. Anal.}, 12(1):138--182, 2002.

\bibitem[McG07]{McGahagan-2007}
Helena McGahagan.
\newblock An approximation scheme for {S}chr\"{o}dinger maps.
\newblock {\em Comm. Partial Differential Equations}, 32(1-3):375--400, 2007.

\bibitem[MRR13]{MRR-blowup-2013}
Frank Merle, Pierre Rapha\"{e}l, and Igor Rodnianski.
\newblock Blowup dynamics for smooth data equivariant solutions to the critical
  {S}chr\"{o}dinger map problem.
\newblock {\em Invent. Math.}, 193(2):249--365, 2013.

\bibitem[MS73]{MS-Sobolev-1973}
James~H. Michael and Leon~M. Simon.
\newblock Sobolev and mean-value inequalities on generalized submanifolds of
  {$R^{n}$}.
\newblock {\em Comm. Pure Appl. Math.}, 26:361--379, 1973.

\bibitem[RRS09]{RRS-Global-2009}
Igor Rodnianski, Yanir~A. Rubinstein, and Gigliola Staffilani.
\newblock On the global well-posedness of the one-dimensional {S}chr\"{o}dinger
  map flow.
\newblock {\em Anal. PDE}, 2(2):187--209, 2009.

\bibitem[Sha12]{Shashikanth-JMP-2012}
Banavara~N. Shashikanth.
\newblock Vortex dynamics in {$\Bbb R^4$}.
\newblock {\em J. Math. Phys.}, 53(1):013103, 21, 2012.

\bibitem[Son17]{Song-Gauss-2017}
Chong Song.
\newblock Gauss map of the skew mean curvature flow.
\newblock {\em Proc. Amer. Math. Soc.}, 145(11):4963--4970, 2017.

\bibitem[SS19]{SS-SMCF-2019}
Chong Song and Jun Sun.
\newblock Skew mean curvature flow.
\newblock {\em Commun. Contemp. Math.}, 21(1):1750090, 29, 2019.

\bibitem[SW18]{SW-uniqueness-2018}
Chong Song and Youde Wang.
\newblock Uniqueness of {S}chr\"{o}dinger flow on manifolds.
\newblock {\em Comm. Anal. Geom.}, 26(1):217--235, 2018.

\end{thebibliography}
\end{document}